\documentclass[a4paper,12pt]{article}
\usepackage{amscd}
\usepackage{amsmath,amsfonts,amssymb,amscd}
\usepackage{indentfirst,graphicx,epsfig}
\usepackage{graphicx,psfrag}
\input{epsf}

\newtheorem{thm}{Theorem}

\newtheorem{lem}{Lemma}
\newtheorem{ob}{Observation}
\newtheorem{claim}{Claim}

\newenvironment {proof} {\noindent{\em Proof.}}{\hspace*{\fill}$\Box$\par\vspace{4mm}}

\def\qed{\hfill \nopagebreak\rule{5pt}{8pt}}

\title{\bf Degree powers in $C_5$-free graphs\footnote{Supported by NSFC and the ``973" program. } }
\author{
\small  Ran Gu, Xueliang Li, Yongtang Shi\\
\small Center for Combinatorics and LPMC-TJKLC \\
\small Nankai University, Tianjin 300071, China \\
\small guran323@163.com, lxl@nankai.edu.cn,  shi@nankai.edu.cn
\date{}}
\begin{document}
\maketitle
\begin{abstract}
Let $G$ be a graph with degree sequence $d_1,d_2,\ldots,d_n$. Given
a positive integer $p$, denote by $e_p(G)=\sum_{i=1}^n d_i^p$. Caro
and Yuster introduced a Tur\'an-type problem for $e_p(G)$: given an
integer $p$, how large can $e_p(G)$ be if $G$ has no subgraph of a
particular type. They got some results for the subgraph of
particular type to be a clique of order $r+1$ and a cycle of even
length, respectively. Denote by $ex_p(n,H)$ the maximum value of
$e_p(G)$ taken over all graphs with $n$ vertices that do not contain
$H$ as a subgraph. Clearly, $ex_1(n,H)=2ex(n,H)$, where $ex(n,H)$
denotes the classical Tur\'an number. In this paper, we consider
$ex_p(n, C_5)$ and prove that for any positive integer $p$ and
sufficiently large $n$, there exists a constant $c=c(p)$ such that
the following holds: if $ex_p(n, C_5)=e_p(G)$ for some $C_5$-free
graph $G$ of order $n$, then $G$ is a complete bipartite graph
having one vertex class of size $cn+o(n)$ and the other
$(1-c)n+o(n)$.
\\[2mm]
[2mm] \textbf{AMS Subject Classification (2010):} 05C35, 05C07.
\end{abstract}

\section{Introduction}

All graphs considered here are finite, undirected, and have no loops
or multiple edges. For standard graph-theoretic notation and
terminology, the reader is referred to \cite{B}. Denote by $ex(n,H)$
the classical Tur\'{a}n number, i.e., the maximum number of edges
among all graphs with $n$ vertices that do not contain $H$ as a
subgraph. Denote by $T_r(n)$ the $r$-partite Tur\'{a}n graph of
order $n$, namely, $ex(n,K_{r+1})=e(T_r(n))$. Given a graph $G$
whose degree sequence is $d_1,\ldots,d_n$, and for a positive
integer $p$, let $e_p(G)=\sum\limits_{i{\rm{ = }}1}^n {{d_i}^p}$.
Caro and Yuster \cite{CY} introduced a Tur\'an-type problem for
$e_p(G)$: given an integer, how large can $e_p(G)$ be if $G$ has no
subgraph of a particular type. Denote by $ex_p(n,H)$ the maximum
value of $e_p(G)$ taken over all graphs with $n$ vertices that do
not contain $H$ as a subgraph. Clearly, $ex_1(n,H)=2ex(n,H)$. It is
interesting to determine the value of $ex_p(n,H)$ and the
corresponding extremal graphs. In \cite{CY}, Caro and  Yuster
considered $K_{r+1}$-free graphs and proved that
\begin{equation}\label{eq1}
    e{x_p}(n,{K_{r + 1}}){\rm{ }} = {\rm{ }}{e_p}({T_r}(n)){\rm{ }}
\end{equation}
for $1 \leq p \leq 3$.

Therefore, it is interesting to find the values of $p$ for which
equality \eqref{eq1} holds and determine the asymptotic value of
$ex_p(n,K_{r+1})$ for large $n$. In \cite{BN}, Bollob\'as and
Nikiforov showed that for every real $p$ ($1\leq p< r$) and
sufficiently large $n$, if $G$ is a graph of order $n$ and has no
clique of order $r+1$, then $ex_p(n,K_{r+1})=e_p(T_r(n))$, and for
every $p\geq r+\lceil\sqrt{2r}\rceil$ and sufficiently large $n$,
$ex_p(n,K_{r+1})>(1+\epsilon)e_p(T_r(n))$ for some positive
$\epsilon=\epsilon(r)$. In \cite{BN1}, Bollob\'as and Nikiforov
proved that if $e_p(G)>(1-1/r)^pn^{p+1}+C$, then $G$ contains more
than $\frac{Cn^{r-p}}{p2^{6r(r+1)+1}r^r}$ cliques of order $r+1$.
Using this statement, they strengthened the Erd\"os--Stone theorem
by using $e_p(G)$ instead of the number of edges.

When considering cycles as the forbidden subgraphs, Caro and Yuster
\cite{CY} determined the value of $e{x_2}(n,{C^ * })$ for
sufficiently large $n$, where $C^ *$ denotes the family of cycles
with even length. And they also characterized the unique extremal
graphs. In \cite{N}, Nikiforov proved that for any graph $G$ with
$n$ vertices, if $G$ does not contain $C_{2k+2}$, then for every
$p\geq 2$, ${e_p}(G) \leq kn^p+O(n^{p-1/2})$.  Since the graph
${K_k}{\rm{ + }}{\overline K _{n{\rm{ - }}k}}$, i.e., the join of
${K_k}$ and $ {\overline K _{n{\rm{ - }}k}}$ contains no $C_{2k+2}$,
that gives $ e{x_p}(n,{C_{2k{\rm{ + }}2}})$, hence $
e{x_p}(n,{C_{2k{\rm{ + }}2}})=kn^p(1+o(1))$, which settles a
conjecture of Caro and Yuster.

In this paper, we will study $ex_p(n, C_5)$. For a fixed
$(r+1)$-chromatic graph $H$, Bollob\'as and Nikiforov \cite{BN1}
showed that for every $r\geq 2$ and $p>0$,
$ex_p(n,H)=e_p(n,K_{r+1})+o(n^{p+1})$. This gives us that $ex_p(n,
C_5)=e_p(n,K_{3})+o(n^{p+1})$. Our main result is the following
theorem.
\begin{thm}\label{thm}
For any positive integer $p$ and sufficiently large $n$, there
exists a constant $c=c(p)$ such that the following holds: if
$ex_p(n, C_5)=e_p(G)$ for some $C_5$-free graph $G$ of order $n$,
then $G$ is a complete bipartite graph having one vertex class of
size $cn+o(n)$ and the other of size $(1-c)n+o(n)$.
\end{thm}

\section{Proof of Theorem \ref{thm} }
When $p=1$, it is  a well-known result of the classical Tur\'{a}n
problem. So in the following we assume $p \ge 2$. Throughout the
paper, let $G$ be the extremal graph satisfying that $ex_p(n,
C_5)=e_p(G)$. Observe that $T_2(n)$ contains no $C_5$, and we have
$$e_p(T_2(n))=\left\lfloor {\frac{n}{{\rm{2}}}} \right\rfloor
{\left( {\left\lceil {\frac{n}{2}} \right\rceil } \right)^p}{\rm{ +
}}\left\lceil {\frac{n}{2}} \right\rceil {\left( {\left\lfloor
{\frac{n}{2}} \right\rfloor } \right)^p}{\rm{ = }}{\left(
{\frac{1}{2}} \right)^p}{n^{p{\rm{ + }}1}}{\rm{ + }}o\left(
{{n^{p{\rm{ + }}1}}} \right).$$ By the definition of $ex_p(n, C_5)$,
we have $e_p(G) \geq e_p(T_2(n))$. Hence, the coefficient of
$n^{p+1}$ in $e_p(G)$ must be at least ${\left( {\frac{1}{2}}
\right)^p}$.

\begin{lem}\label{lem1}
For every integer $p$ and sufficiently large $n$, if $e_p(G)=ex_p(n,
C_5)$, then $\Delta(G)=an+o(n)$, where the constant $a=a(p)\geq
{\frac{1}{2}}$.
\end{lem}

\noindent{\it Proof.} Suppose $\Delta(G)=o(n)$, we then have $e_p(G)
\leq n \cdot [\Delta(G)]^p=n \cdot o(n^p) = o(n^{p+1})$, a
contradiction. Let $\Delta(G)=an+o(n)$. Then we have $e_p(G) \leq n
\cdot (an)^p+o(n^{p+1})=a^p n^{p+1}+o(n^{p+1})$, which implies $a
\geq {\frac{1}{2}}$.\qed

In order to describe the structure of the extremal graph $G$, we
introduce some classes of graphs and a graph operation on two or
more graphs. Let $S^k$ denote the set of graphs of order $k$ as
shown in Figure \ref{figure1}. And graphs $S_1,S_2,S_3$ are also
shown in Figure \ref{figure1}. Each of these graphs has a labeled
vertex, i.e., the cross vertex as shown in Figure \ref{figure1}.

\begin{figure}[!hbpt]
\begin{center}
\includegraphics[scale=0.4]{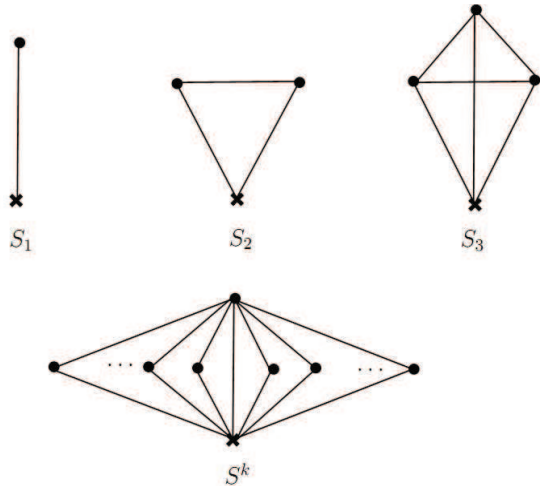}
\end{center}
\caption{The illustration of $S^k$ and $S_i$.}\label{figure1}
\end{figure}

Let $\mathcal{S}=\{S_1,S_2,S_3\}$, $\mathcal{S}^\ast=\mathcal{S}
\cup S^4\cup S^6\cup \cdots$, for all possible $k$. When we say
``$\textit{attaching}$" two graphs in $\mathcal{S}^\ast$, it means
that we identify the labeled vertices in each graph. Note that this
attaching operation could be applied on more than two graphs. Before
the proof, we recall a classical result of Erd\"{o}s and Gallai
\cite{EG}.

\begin{lem}\label{lem2}
If a graph of order $n$ has more than $kn/2$ edges, then it contains
a path of order $k+2$. \qed
\end{lem}
\noindent {\bf Proof of Theorem \ref{thm}:}\ We will consider the
following two cases.

\noindent {\bf Case 1.} For any vertex $u$ with maximum degree in
$G$, there is no edge in $G[N_G(u)]$.

In this case, we can construct a complete bipartite graph $H$, which
satisfies that $e_p(H) \geq e_p(G)$. The complete bipartite graph
$H=(X,Y)$ can be constructed as follows: $X=(V(G) \setminus N_G(u))
\cup \{u\}$ and $Y=N_G(u)$. It is easy to check that $d_H(v) \geq
d_G(v)$ for any vertex $v\in V(G)$, hence  $e_p(H) \geq e_p(G)$.
Since $G$ is the extremal graph, we can deduce that $G$
itself is isomorphic to $H$.

\noindent {\bf Case 2.}  There exists a vertex $v$ with maximum
degree in $G$, such that there is at least one edge in $G[N_G(v)]$.

Let $u$ be such a vertex with maximum degree. By Lemma \ref{lem1},
we assume that $d_G(u)=an+o(n)$, where $a\geq {\frac{1}{2}}$. Let
$A$ denote the set $\{u\}\bigcup N_G(u)$, and $B$ denote
$V(G)\setminus (\{u\}\bigcup N_G(u))$, respectively. Since $G$ is
$C_5$-free, we have that $G[A]$ is also $C_5$-free. Then we can get
that $G[A]$ must be constructed by attaching some graphs in
$\mathcal{S}^\ast$, and moreover, $u$ is just the vertex identified
by labeled vertices. For example, $G[A]$ may be isomorphic to the
graph as shown in Figure \ref{figure2}.

\begin{figure}[!hbpt]
\begin{center}
\includegraphics[scale=0.8]{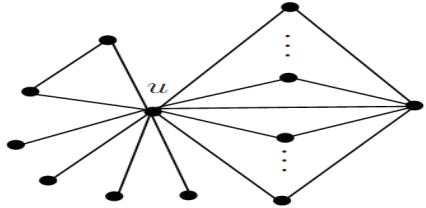}
\end{center}
\caption{An example of $G[A]$.}\label{figure2}
\end{figure}

In fact, considering the edges between $A$ and $B$, we can obtain
the following two observations. Note that the vertices in $B$ can
only be adjacent to the unlabeled vertices, since all of neighbors
of $u$ are in $A$. Without loss of generality, suppose $G[A]$ is
constructed by attaching $t_i$ $S_i$'s, $i=1,2,3$ and $r_k$ $S^k$'s,
for possible $k$. Observe that if $w\in B$, then $w$ cannot be
adjacent to two graphs among all $t_i$ $S_i$'s, and $r_k$ $S^k$'s,
except one case that the two graphs are $S_1$ and $S_1$.

\begin{ob}
For any vertex $w$ in $B$, the edges between $w$ and $A$ can only be
one of the following four cases:

(a)  $w$ is not adjacent to any vertex in $A$.

(b)  $w$ is adjacent to some unlabeled vertices of $S_1$'s;

(c)  $w$ is adjacent to one or two unlabeled vertices of exactly one
$S_2$;

(d)  $w$ is adjacent to only one graph $F$ among all $t_i$ $S_i$'s,
and $r_k$ $S^k$'s. Moreover, $w$ is adjacent to exactly one
unlabeled vertex in $F$.
\end{ob}

\begin{ob}
For any edge $w_1w_2$ in $G[B]$, the edges between $w_1$ ($w_2$) and
$A$ can only be one of the following two cases:

(a)  $w_1$ and $w_2$ are adjacent to the same unlabeled vertex in
exactly one graph among all $t_i$ $S_i$'s and $r_k$ $S^k$'s;

(b)  one of $w_1$ and $w_2$, say $w_1$,  is adjacent to no vertices
in $A$, $w_2$ is adjacent to vertices in $A$ as described of
Observation 1.
\end{ob}

With the aid of the above two observations and the assumption of
$G$, we can prove the following claim.

\begin{claim}\label{claim1}
$G[A]$ is isomorphic to the graph obtained by attaching one $S_2$
and $d_G(u)-2$ $S_1$'s.
\end{claim}

\begin{proof}
Let $\mathcal{A}=\{S\in \{S_2,S_3,S^k\}: S\subseteq G[A]$ and some
unlabeled vertex $v$ in $S$ has degree $d_G(v)=O(n)\}$. By the
previous observations, we have $|\mathcal{A}|=o(n)$, since
$|A|=an+o(n)$, $|B|=(1-a)n+o(n)$ and the number of edges between the
vertices in $\mathcal{A}$ and $B$ will be no more than $|B|$. So
$\sum\limits_{v \in V{\rm{(}}S{\rm{)}},\  S \in \mathcal{A} }
{{d_G}^p\left( v \right)} =o(n^{p+1})$. Therefore, the vertices in
$\mathcal{A}$ have no contribution to the value of the coefficient
of $n^{p+1}$ in $e_p(G)$. In order to maximize the value of
$e_p(G)$, $G[A]$ must consist of as many $S_1$'s as possible. Since
we assume that there exists at least one edge in $G[A]$, $G[A]$ must
be isomorphic to the graph obtained by attaching of one $S_2$ and
$d_G(u)-2$ $S_1$'s.
\end{proof}

Let $A_1$ denote the set of all the unlabeled vertices in $S_1$
contained in $G[A]$.

From Claim \ref{claim1}, the extremal graph in Case 1 satisfies the
description in Claim \ref{claim1}. We construct two graphs $G'$ and
$G^\ast$ to characterize the extremal graph $G$ in detail. Let
$V(G')=V(G^\ast)=V(G)$, both $G'$ and $G^\ast$ satisfy the
assumption of Case 2 and the description of Claim \ref{claim1},
i.e., in both $G'$ and $G^\ast$, let $u$ be the vertex with maximum
degree $d_G(u)$, there exist edges in $G[N_G'(u)]$ and
$G[N_G^\ast(u)]$. Without loss of generality, let $A=\{u\}\bigcup
N_G'(u)=\{u\}\bigcup N_G^\ast(u)$, and let
$B=V(G')\setminus(\{u\}\bigcup N_G'(u))=V(G^\ast)\setminus
(\{u\}\bigcup N_G^\ast(u))$. Observe that $G'[A]$ and $G^\ast[A]$
satisfy the description of Claim \ref{claim1}. Hence, we can still
use notation $A_1$ to denote the set of all the unlabeled vertices
in $S_1$ contained in $G'[A]$, and the same set in $G^\ast[A]$.

The difference between $G'$ and $G^\ast$ is as follows.  For $G'$,
$G'[B]$ is empty, every vertex in $B$ is adjacent to every vertex in
$A_1$ and there is no edge between $A\setminus A_1$ and $B$. And for
$G^\ast$, there are two vertices in $B$, say $w_1$, $w_2$, such that
$G^\ast[B]$ is a complete bipartite graph with one class $\{w_1
,w_2\}$, every vertex in $B\setminus \{w_1 ,w_2\}$ is adjacent to
every vertex in $A_1$ and there is no edge between $A\setminus A_1$
and $B$.

The next claim characterizes the extremal graph $G$ in Case 2. Since
we only consider the case that $n$ is sufficiently large, from the
preceding discussions, we can assume that $d_G(u)=an$ instead of
$an+o(n)$ to simplify the calculation.

\begin{claim}\label{claim2}
$e_p(G)$ is equal to either $e_p(G')$ or $e_p(G^\ast)$.
\end{claim}

\begin{proof}
Firstly, we calculate $e_p(G')$ and $e_p(G^\ast)$. For any vertex
$v\in A_1$, $d_{G'}(v)=(1-a)n$; for any vertex $v\in A\setminus
(A_1\cup \{u\})$, $d_{G'}(v)=2$; and for any vertex $v$ in $B$,
$d_{G'}(v)=an-2$. Hence,
\[{e_p}(G') = {\left( {an} \right)^p}{\rm{ + }}2 \times {2^p}{\rm{ + }}(an{\rm{ - }}2)
{[(1{\rm{ - }}a)n]^p}{\rm{ + }}[(1{\rm{ - }}a)n{\rm{ -
}}1]{(an{\rm{ - }}2)^p}.\] Similarly, Observe that for any vertex
$v\in A_1$, $d_{G^\ast}(v)=(1-a)n-2$, and for any vertex $v\in
A\setminus (A_1\cup \{u\})$, $d_{G^\ast}(v)=2$, also we have
$d_{G^\ast}(u)=an$, $d_{G^\ast}(w_i)=(1-a)n-3$, $i=1,2$, and for any
vertex $w$ in $B\setminus \{w_1 ,w_2\}$, $d_{G^\ast}(w)=an$. It is
easy to calculate that $e_p(G^\ast)= \left[ {\left( {1 - a}
 \right)n - 2} \right]
{\left( {an} \right)^p} + \left( {an - 2} \right){\left[ {\left( {1
- a} \right)n - 2} \right]^p} + 2{\left[ {\left( {1 - a} \right)n -
3} \right]^p} + 2 \times {2^p}$.

We assume $e_p(G)> e_p(G')$. Then, there must exist some vertex $v$
satisfying $d_G(v)> d_{G'}(v)$. Note that  for each vertex $v\in
A_1$, the degree of $v$ is at most $(1-a)n$, we only need to
consider such two cases.

{\bf Case 3.} There exists some vertex $v\in B$, such that $d_G(v)>
d_{G'}(v)$, and for each vertex $v'\in A$, $d_G(v')$ is no larger
than $d_{G'}(v')$.

Let $B_1=N_G(v)\cap B$. In the following, we will consider the
following two subcases.

{\bf Subcase 3.1.} $|N_G(v)\cap A|=xn$ and $|B_1|=yn$, where $0\leq
x < a$, $0<y\leq 1-a$, $x+y\leq a$, and if $x=0$, $|N_G(v)\cap
A|\geq 2$.

Firstly, we know that $G[B_1]$ contains no path of order $4$, since
otherwise, there will exist one $C_5$ including $v$.

By Lemma \ref{lem2}, we have that the number of edges in $G[B_1]$,
denote by $e(G[B_1])$, is no more than $2yn/2=yn$. Hence,
$e(G[B_1])=\sum\limits_{v \in {B_1}}{{d_{G[{B_1}]}}(v)}\leq 2yn$. We
will calculate the maximum possible value of $e_p(G)$. We assume
that there is one vertex in $B_1$ with degree $d_{G[{B_1}]}=yn-1$
and the remaining vertices in $B_1$ with degree $d_{G[{B_1}]}=1$.
For each vertex in $B_1$, we can assume that it is adjacent to each
vertex in $B\setminus B_1$. (Note that the vertices in $B_1$ cannot
be adjacent to the vertices in $A$ from the previous observations.)
Suppose that all the vertices in $B\setminus B_1$ reach the maximum
degree $an$ in $G$. We can see that such a situation can maximize
the value of $e_p(G)$, and it may be much larger than the exact
value of $e_p(G)$. We then have
\begin{eqnarray*}
e_p(G) &\leq &  {\left( {an} \right)^p}{\rm{ + }}2 \times {2^p}{\rm{
+ }}{\left( {an} \right)^p}{\rm{ + }}\left[ {\left( {1{\rm{ - }}a} \right)n{\rm{
- }}2{\rm{ - }}yn} \right]{\left( {an} \right)^p}\\
\null  &\null&   +\left( {an{\rm{ - }}2} \right){\left[ {\left(
{1{\rm{ - }}a} \right)n{\rm{ - }}yn} \right]^p}{\rm{ +
}}\left( {yn{\rm{ - }}1} \right){\left[ {\left( {1{\rm{ - }}a}
\right)n{\rm{ - }}yn} \right]^p}{\rm{ + }}{\left[ {\left( {1{\rm{ -
}}a} \right)n{\rm{ - }}2} \right]^p}.
\end{eqnarray*}
Expanding the right hand side of the inequality above, the
coefficient of $n^{p+1}$ is $$\left( {1{\rm{ - }}a{\rm{ - }}y}
\right){a^p}{\rm{ + }}a{\left( {1{\rm{ - }}a{\rm{ - }}y}
\right)^p}{\rm{ + }}y{\left( {1{\rm{ - }}a{\rm{ - }}y}
\right)^p}{\rm{ = }}\left( {y{\rm{ + }}a} \right){\left( {1{\rm{ -
}}a{\rm{ - }}y} \right)^p}{\rm{ + }}\left( {1{\rm{ - }}a{\rm{ - }}y}
\right){a^p}.$$ Since from the previous calculation we know that the
coefficient of $n^{p+1}$ in $e_p(G')$ is $a{\left( {1{\rm{ - }}a}
\right)^p}{\rm{ + }}{a^p}\left( {1{\rm{ - }}a} \right)$, to derive a
construction to our assumption, it is sufficient to show that
\[\left( {y{\rm{ +}}a} \right){\left( {1{\rm{ - }} a{\rm{ - }}y}
\right)^p}{\rm{ +}}\left( {1{\rm{ - }}a{\rm{ - }}y}
 \right){a^p} <  a{\left( {1{\rm{
- }}a} \right)^p}{\rm{ + }}{a^p}\left( {1{\rm{ - }}a} \right)\] for
sufficiently large $n$. Let $$f(a,y)=a{\left( {1{\rm{ - }}a}
\right)^p}{\rm{ + }}{a^p}\left( {1{\rm{ - }}a} \right){\rm{ -
}}\left[ {\left( {y{\rm{ + }}a} \right){{\left( {1{\rm{ - }}a{\rm{ -
}}y} \right)}^p}{\rm{ + }}\left( {1{\rm{ - }}a{\rm{ - }}y}
\right){a^p}} \right].$$ We will show that $f(a,y)>0$. We first
suppose that $\frac{{1{\rm{ - }}a}}{2} < y < 1 - a$, i.e., $1-a-y<y<a$.
Then, we have
\begin{eqnarray*}
f(a,y)  &=&    a{\left( {1{\rm{ - }}a} \right)^p}{\rm{ - }}y{\left( {1{\rm{
- }}a{\rm{ - }}y} \right)^p}{\rm{ - }}a{\left( {1{\rm{ - }}a{\rm{ - }}y} \right)^p}{\rm{ + }}{a^p}y\\
\null   &>&    a{\left( {1{\rm{ - }}a} \right)^p}{\rm{ - }}y{\left( {1{\rm{
- }}a{\rm{ - }}y} \right)^p}{\rm{ - }}a{y^p}{\rm{ + }}{a^p}y\\
\null   & > & a{y^p}{\rm{ - }}y{\left( {1{\rm{ - }}a{\rm{ - }}y}
\right)^p}{\rm{ - }}a{y^p}{\rm{ + }}{a^p}y=   y{a^p}{\rm{ -
}}y{\left( {1{\rm{ - }}a{\rm{ - }}y} \right)^p} >   0.
\end{eqnarray*}
Now we suppose $0<y \le \frac{{1{\rm{ - }}a}}{2}$. In this case, we
have
\begin{eqnarray*}
f(a,y)  & = &   a{\left( {1{\rm{ - }}a} \right)^p}{\rm{ - }}\left( {y{\rm{
+ }}a} \right){\left( {1{\rm{ - }}a{\rm{ - }}y} \right)^p}{\rm{ + }}{a^p}y\\
\null   &\geq & a{\left( {1{\rm{ - }}a} \right)^p}{\rm{ + }}{a^p}y{\rm{
- }}\frac{{1{\rm{ + }}a}}{2}{\left( {1{\rm{ - }}a{\rm{ - }}y} \right)^p}\\
\null   &=&  a{\left( {1{\rm{ - }}a} \right)^p}{\rm{ + }}{a^p}y{\rm{
- }}\frac{{1{\rm{ + }}a}}{2}{\left( {1{\rm{ - }}a} \right)^p}{\rm{
+ }}\frac{{1{\rm{ + }}a}}{2}{\left( {1{\rm{ - }}a} \right)^p}{\rm{
- }}\frac{{1{\rm{ + }}a}}{2}{\left( {1{\rm{ - }}a{\rm{ - }}y} \right)^p}\\
\null   &=&  {a^p}y{\rm{ + }}\frac{{a{\rm{ - }}1}}{2}{\left( {1{\rm{
- }}a} \right)^p}{\rm{ + }}\frac{{1{\rm{ + }}a}}{2}\left[ {{{\left( {1{\rm{
- }}a} \right)}^p}{\rm{ - }}{{\left( {1{\rm{ - }}a{\rm{ - }}y} \right)}^p}} \right]\\
\null   &>&  {a^p}y{\rm{ + }}\frac{{a{\rm{ - }}1}}{2}{\left( {1{\rm{
- }}a} \right)^p}{\rm{ + }}\frac{{1{\rm{ + }}a}}{2}{y^p}>
{a^p}y{\rm{ + }}\frac{{a{\rm{ - }}1}}{2}{\left( {1{\rm{
- }}a} \right)^p}{\rm{ + }}\frac{{1{\rm{ - }}a}}{2}{y^p}\\
\null   &=&   {a^p}y{\rm{ + }}\frac{{1{\rm{ - }}a}}{2}\left[
{{y^p}{\rm{ - }}{{\left( {1{\rm{ - }}a} \right)}^p}} \right]\geq
y\left[ {{a^p}{\rm{ + }}{y^p}{\rm{ - }}{{\left( {1{\rm{ - }}a}
\right)}^p}} \right]>   0.
\end{eqnarray*}

Hence, we have proved that $f(a,y)>0$.

{\bf Subcase 3.2.} $|N_G(v)\cap A|=an-o(n)$ and $|B_1|=o(n)$.

With similar methods, we have
\begin{eqnarray*}
e_p(G) &\leq &  {\left( {an} \right)^p}{\rm{ + }}2 \times {2^p}{\rm{
+ }}{\left( {an} \right)^p}{\rm{ + }}\left[ {\left( {1{\rm{
- }}a} \right)n{\rm{ - }}2{\rm{ - }}o{\rm{(}}n{\rm{)}}} \right]{\left( {an} \right)^p}\\
\null  &\null& {\rm{ + }}\left( {an{\rm{ - }}2} \right){\left[ {\left( {1{\rm{
- }}a} \right)n{\rm{ - }}o{\rm{(}}n{\rm{)}}} \right]^p}
{\rm{ + }}\left( {o{\rm{(}}n{\rm{) - }}1} \right){\left[ {\left( {1{\rm{
 - }}a} \right)n{\rm{ - }}o{\rm{(}}n{\rm{)}}} \right]^p}{\rm{
 + }}{\left[ {\left( {1{\rm{ - }}a} \right)n{\rm{ - }}2} \right]^p}.
\end{eqnarray*}
Similarly, there are two cases when we compare the values of
$e_p(G)$ and $e_p(G')$.

$\bullet$ The  $o(n)$ part of $|N_G(v)\cap A| $, denoted  by
$\omega$, satisfies that $\omega\rightarrow +\infty$.

Observe that $n^p<\omega n^p < n^{p+1}$. So we need to consider the
coefficient of $\omega n^p$. By expanding the expression of
$e_p(G)$, it is clear that the coefficient is ${\rm{ - }}{a^p}{\rm{
+ }}{\left( {1{\rm{ - }}a} \right)^p}{\rm{ \leq }}0$, which implies
$e_p(G)\leq e_p(G')$, a contradiction.

$\bullet$ The $o(n)$ part of $|N_G(v)\cap A| $ is a constant.

Let $o(n)=c$, $c\geq 1$. We will prove in that subcase, $G$ is
isomorphic to $G^\ast$. Now we consider the structure of $G$. If a
vertex in $B\setminus B_1$ has degree $an$, then at least two of its
neighbors in $B$ will be not adjacent to any vertices in $A$. So in
order to maximize the number of vertices whose degree is $O(n)$, we
suppose that as many as possible vertices in $B\setminus B_1$ have
degree $an$, all of them have only two neighbors in $B$. It is not
difficult to get that if they share two common neighbors in $B$, we
will have a larger value of $e_p(G)$. Furthermore, let these two
common neighbors be both in $B_1$, and there are no other vertices
in $B_1$, we can get the maximum value of $e_p(G)$ in that
situation. And we can see that $c$ is equal to $2$ in such case.
Moreover, $G$ is isomorphic to $G^\ast$.

{\bf Subcase 3.3.} $|N_G(v)\cap A|=1$.

Since $a\geq {\frac{1}{2}}$, $|B|=(1-a)n-1$,  and we assume that
$d_G(v)>d_{G'}(v)=an-2$, we have that $a={\frac{1}{2}}$ and
$|B_1|=(1-a)n-2$, i.e., $v$ is adjacent to every vertex in
$B\setminus \{v\}$. Let $N_G(v)\cap A=\{v'\}$, by Observation 2, the
vertices in $B$ can only be adjacent to $v'$ in $A$. To maximize the
value of $e_p(G)$, let all the vertices in $B$ be adjacent to $v'$
and $G[B]$ be a complete graph. Note that every vertex has its
maximal possible degree. Hence, $e_p(G)\leq 2 \times {2^p} +
{(an)^p} + {\left[ {\left( {1 - a} \right)n}
 \right]^p}
+ \left( {an - 3} \right) + \left[ {\left( {1 - a} \right)n - 1} \right]
{\left[ {\left( {1 - a} \right)n - 1} \right]^p} = {\left( {\frac{1}{2}}
\right)^{p + 1}}{n^{p + 1}} + o\left( {{n^{p + 1}}} \right) <
{e_p}\left( {{T_2}\left( n \right)} \right)$, a contradiction.

{\bf Case 4.} There exists a vertex $v\in A\setminus (A_1\cup
\{u\})$ such that $d_G(v)> d_{G'}(v)$.

Let $A\setminus (A_1\cup \{u\})=\{v_1,v_2\}$.  Without loss of
generality, assume that $d_G(v_1)=2+x$,  $d_G(v_2)=2+y$. Suppose
that $w\in B$ is adjacent to $v_1$, from Observation 2, $w$ can not
be adjacent to any vertices in $A_1$, and to avoid $5$-cycles, the
neighbors of $w$ in $B$ can not be adjacent to any vertices in
$A_1$. Just similar to Case 3, we can derive that there are two
vertices $w'$, $w''$ in $B$, such that all neighbors of $v_1$ in $B$
is adjacent to $w'$,  and all neighbors of $v_2$ in $B$ is adjacent
to $w''$,  the set of remaining vertices in $B$ and $\{w', w''\}$
form a complete bipartite graph. Note that $v_1$ and $v_2$ have no
common neighbors in $B$ in order to avoid $5$-cycles and maximize
the value of $e_p(G)$. If either $x$ or $y$ is zero, then $w'=w''$.
So, if $x\geq 1$, $y\geq 1$, then,
\[{e_p}(G)={\left( {2 + x} \right)^p} + {\left( {2 + y}
\right)^p} + {\left( {an} \right)^p} + (x + y) \cdot {2^p} +
\left( {an - 2} \right){\left[ {\left( {1 - a} \right)n - 2 - x - y}
\right]^p}\]
\begin{equation}\label{eq2}
+ \left[ {\left( {1 - a} \right)n - 3 - x - y} \right]
{\left( {an} \right)^p} + {\left[ {\left( {1 - a} \right)n - 3 - y}
\right]^p} + {\left[ {\left( {1 - a} \right)n - 3 - x} \right]^p}.
\end{equation}
Suppose either $x$ or $y$ is zero, by symmetry, we need only
consider the case when $y=0$ and $x\geq 1$. In such case, we have
\[{e_p}\left( G \right) = x \cdot {2^p} + {2^p} + {\left( {2 + x}
\right)^p} + {\left( {an} \right)^p} + \left( {an - 2} \right)
{\left[ {\left( {1 - a} \right)n - 1 - x} \right]^p}\]
\begin{equation}\label{eq3}
+\left[ {\left( {1 - a} \right)n - 2 - x} \right]{\left( {an - 1}
\right)^p} + {\left[ {\left( {1 - a} \right)n - 2} \right]^p}.
\end{equation}

In equation (\ref{eq2}), if $x$ or $y$ is $O(n)$, then the
coefficient of $n^{p+1}$ is strictly less than $a{\left( {1{\rm{ -
}}a} \right)^p}{\rm{ + }}{a^p}\left( {1{\rm{ - }}a} \right)$. Since
the coefficient of $n^{p+1}$ in $e_p(G')$ is $a{\left( {1{\rm{ -
}}a} \right)^p}{\rm{ + }}{a^p}\left( {1{\rm{ - }}a} \right)$, we
have $e_p(G)< e_p(G')$, which contradicts to our assumption. Hence,
$x$ and $y$ are both $o(n)$, and ${\left( {2 + x} \right)^p} +
{\left( {2 + y}\right)^p}$ has no contribution to the coefficient of
$n^p$. Thus,  the coefficient of $n^p$ in $e_p(G)$ is
\begin{eqnarray*}
\null &\null & a^p - 2{\left( {1 - a} \right)^p} - pa\left( {2 + x + y}
\right){\left( {1 - a} \right)^{p - 1}} - {a^p}\left( {3 + x + y}
\right) + 2{\left( {1 - a} \right)^p}\\
\null  &=&   - pa\left( {2 + x + y} \right){\left( {1 - a}
\right)^{p - 1}}-{a^p}\left( {2 + x + y} \right).
\end{eqnarray*}
From the expression of $e_p(G^\ast)$, the coefficient of $n^p$ in $e_p(G^\ast)$
is $ - 2pa{\left( {1 - a} \right)^{p - 1}} - 2{a^p}$, which is larger than
$- pa\left( {2 + x + y}
\right){\left( {1 - a} \right)^{p - 1}}
- {a^p}\left( {2 + x + y} \right)$. Similarly, when $y=0$, we can deduce that $x$ is $o(n)$.
With some calculations, one can see
that the coefficient of $n^p$ in (\ref{eq3}) is less than that in $e_p(G^\ast)$.
Hence, $e_p(G)< e_p(G^\ast)$
for sufficiently large $n$, i.e.,
$G$ can not be the extremal graph, a contradiction.

Combining all cases above, we have proved this claim.
\end{proof}

In the sequel, we will prove that the extremal graph described in
Case 2 will always have a smaller value of $e_p(\cdot)$ than the
extremal graph in Case 1. Let $G_1$ and $G_2$ be the extremal graph
in Case 1 and Case 2, respectively. So we have $e_p(G_2)=\max\{
e_p(G'),e_p(G^\ast)\}$. It is easy to get that, the coefficient of
$n^{p+1}$ in the expression of $e_p(G_2)$ is $a{\left( {1{\rm{ -
}}a} \right)^p}{\rm{ + }}{a^p}\left( {1{\rm{ - }}a} \right)$, which
is equal to that of $e_p(G_1)$. The coefficient of $n^p$ in the
expression of $e_p(G')$ is
\begin{eqnarray*}
&\null& {a^p} - 2p\left( {1 - a} \right){a^{p - 1}} - 2{\left( {1 - a} \right)^p} - {a^p}
= {\rm{ - }}2p\left( {1{\rm{ - }}a} \right){a^{p{\rm{ - }}1}}{\rm{ - }}2{\left( {1{\rm{ - }}a} \right)^p} < 0.
\end{eqnarray*}
And the coefficient of $n^p$ in the
expression of $e_p(G^\ast)$ is $- 2{a^p} -
2pa{\left( {1 - a} \right)^{p - 1}}< 0$.

Therefore, for sufficiently large $n$, $e_p(G_2)<\left[a{\left( {1{\rm{ -
}}a} \right)^p}{\rm{ + }}{a^p}\left( {1{\rm{ - }}a} \right)\right]n^{p+1}$,
i.e., $e_p(G_2)<e_p(G_1)$.

In conclusion, if $ex_p(n, C_5)=e_p(G)$ for some $C_5$-free graph
$G$ of order $n$, then $G$ is isomorphic to $G_1$. Hence $G$ is a
complete bipartite graph. Moreover, the size of one class is
$cn+o(n)$ and the other is $(1-c)n+o(n)$, where $c$ maximizes the
function $f(x)=x{\left( {1{\rm{ - }}x} \right)^p}{\rm{ +
}}{x^p}\left( {1{\rm{ - }}x} \right)$ in $\left[ {\frac{1}{2}{\rm{ ,
}}\ 1} \right]$.\qed

\end{document}